\documentclass[11pt]{amsart}
\usepackage{latexsym, mathrsfs, color, tikz, multirow, bbm, mathtools}
\usepackage[vcentermath]{youngtab}
\usepackage{young}
\usepackage{graphics}
\usepackage{subcaption}
\input{xy}
\xyoption{all}

\usepackage[colorlinks=true, pdfstartview=FitV,
 linkcolor=black,citecolor=black,urlcolor=black]{hyperref}
 
\usepackage{tikz}

\numberwithin{equation}{section}
\theoremstyle{plain}
\newtheorem{Thm}[equation]{Theorem}

\newtheorem{Cor}[equation]{Corollary}
\newtheorem{Lem}[equation]{Lemma}

\theoremstyle{definition}
\newtheorem{Def}[equation]{Definition}
\newtheorem{Exa}[equation]{Example}
\newtheorem{Rmk}[equation]{Remark}

\newenvironment{red}{\relax\color{red}}{\relax}
\newenvironment{blue}{\relax\color{blue}}{\hspace*{.5ex}\relax}

\newcommand{\ber}{\begin{red}}
\newcommand{\er}{\end{red}}
\newcommand{\beb}{\begin{blue}}
\newcommand{\eb}{\end{blue}}

\begin{document}

\title[Reverse Permutations in the RSK Correspondence]{Permutations whose Reverse Shares the Same Recording Tableau in the RSK correspondence}

\author[T. J. Ervin]{Tucker J. Ervin}
\email{tjervin@crimson.ua.edu}

\author[B. Jackson]{Blake Jackson}
\email{bajackson9@crimson.ua.edu}

\author[J. Lane]{Jay Lane}
\email{Jay.D.Lane-1@ou.edu}

\author[K. Lee]{Kyungyong Lee}
\email{klee94@ua.edu; klee1@kias.re.kr}

\author[S. D. Nguyen]{Son Dang Nguyen}
\email{sdnguyen1@crimson.ua.edu}

\author[J. O'Donohue]{Jack O'Donohue}
\email{jro5347@psu.edu}

\author[M. Vaughan]{Michael Vaughan}
\email{Michael.Vaughan@briarcliff.edu}
\thanks{All authors were supported by the University of Alabama. KL was supported by the NSF grant DMS-2042786 and Korea Institute for Advanced Study. }

\begin{abstract}
The RSK correspondence is a bijection between permutations and pairs of standard Young tableaux with identical shape, where the tableaux are commonly denoted $P$ (insertion) and $Q$ (recording). 
It has been an open problem to demonstrate
\begin{center}
\begin{equation*}
|\{w \in \mathfrak{S}_n | \, Q(w) = Q(w^r)\}|
=
\begin{dcases}
	2^{\frac{n-1}{2}}{n-1 \choose \frac{n-1}{2}} & n \text{ odd} \\
	0 & n \text{ even} \\
\end{dcases},
\end{equation*}
\end{center}
where $w^r$ is the reverse permutation of $w$.
First we show that for each $w$ where $Q(w) = Q(w^r)$ the recording tableau $Q(w)$ has a symmetric hook shape and satisfies a certain simple property.
From these two results, we succeed in proving the desired identity.

\end{abstract}

\maketitle

\section{Introduction}

First described in 1938 by Robinson \cite{Ro}, the bijection began as the Robinson--Schensted (RS) correspondence, and Robinson used it in an attempt to prove the Littlewood--Richardson rule. 
In 1961, Schensted \cite{Sc} gave a much simpler description of the algorithm, and it is the one we will use throughout our paper. 
Even though the two descriptions are very different, the correspondence usually credits both authors. 
Knuth \cite{K} extended the RS correspondence in 1970 to one between non-negative integer matrices and semi-standard Young tableaux. 
The generalized bijection is referred to as the Robinson--Schensted--Knuth (RSK) correspondence. 
While Knuth's formulation and results are important to the theory and have been widely used, we keep our focus solely on permutations, which we write using one line notation. 

Given a permutation $w$, there are three operations we can perform: the reverse $w^r$, the complement $w^c$, and the inverse $w^{-1}$.
\begin{Def} \label{def-permutations}
Let $w = w_1\dots w_n \in \mathfrak{S}_n$. Then we define the {\bf reverse permutation} $w^r = w_n\dots w_1$, the {\bf complement permutation} $w^c = (n+1-w_1)\dots (n+1-w_n)$, and the {\bf reverse-complement permutation} $w^{rc} = w^{cr} = (n+1-w_n)\dots (n+1-w_1)$.
\end{Def}
There exist several relations between the recording and insertion tableaux of $w$ and its image under the three operations.
The insertion tableaux of $w$ and $w^r$ are transposes of each other, written as $P(w) = P(w^r)^T$ \cite[Theorem 3.2.3]{Sa}.
Similarly, the recording tableaux are related by $Q(w) = \epsilon(Q(w^r))^T$ \cite[Theorem 3.9.4]{Sa}, where $\epsilon$ is the evacuation map.
Other relations exist with regards to the inverse and complement operations, such as $P(w) = Q(w^{-1})$, $Q(w) = P(w^{-1})$, $P(w) = \epsilon(P(w^c))^T$, and $Q(w) = Q(w^c)^T$ \cite[Theorem 4.1.1]{L}.
To summarize, the RSK correspondence takes $w$ and its images under the operations to the following:
$$RSK(w) = (P(w), Q(w))$$ 
$$RSK(w^c) = (\epsilon(P(w))^T, Q(w)^T)$$
$$RSK(w^r) = (P(w)^T, \epsilon(Q(w))^T)$$
$$RSK(w^{rc}) = (\epsilon(P(w)), \epsilon(Q(w)))$$
$$RSK(w^{-1}) = (Q(w),P(w))$$
$$RSK(w^{-1c}) = (\epsilon(Q(w))^T, P(w)^T)$$
$$RSK(w^{-1r}) = (Q(w)^T, \epsilon(P(w))^T)$$
$$RSK(w^{-1rc}) = (\epsilon(Q(w)), \epsilon(P(w))).$$
Every combination of the three operations reduces to one of the eight options above.
This then brings two interesting questions: what kind of and how many permutations have their recording tableaux fixed by these combinations?

For the complement, there are no non-trivial permutations such that $Q(w) = Q(w^c)$, as $Q(w^c) = Q(w)^T$.
The only possible such permutation is $1 \in \mathfrak{S}_1$.
The set of permutations such that $Q(w) = Q(w^{-1})$ is the set of involutions of $\mathfrak{S}_n$.
Its cardinality is given by $\sum_{\lambda \, \vdash n} f^\lambda$, where $f^\lambda$ is the number of standard Young tableaux of shape $\lambda$.
As $Q(w) = Q(w^{-1 c})$ only when $Q(w) = P(w)^T$, its cardinality is the sum of all $f^{\lambda'}$, where $\lambda'$ is a shape of size $n$ preserved by transposition.
Additionally, it is straightforward to show that the sets of permutations where $Q(w) = Q(w^{-1 rc})$ or $Q(w) = Q(w^{-1 r})$ are respectively equal to the previous two sets under the reverse operation.
This leaves only permutations which have fixed recording tableaux under the reverse and the reverse complement maps to count.

The question of what permutations have fixed recording tableaux under the reverse map may have been long-posed, but we first encountered the problem when using Jeremy L. Martin's ``Lecture Notes on Algebraic Combinatorics'' \cite[Exercise 9.8(b)]{M}.
In this paper we describe and count the permutations, $w$, such that $Q(w) = Q(w^r)$.
Our main theorem is as follows.

\begin{Thm} \label{main-theorem}
Let $w \in \mathfrak{S}_n$.
Then $Q(w) = Q(w^r)$ if and only if $Q(w)$ satisfies both of the following properties:
\begin{itemize}

\item $Q(w)$ has a symmetric hook shape,

\item The element $i$ in the first row of $Q(w)$ implies that $n-i+2$ belongs to the first column of $Q(w)$ for all $i \in [n]$ with $i > 1$.

\end{itemize}
In particular, we have the formula
$$|\{w \in \mathfrak{S}_n | \, Q(w) = Q(w^r)\}| =
\begin{dcases}
	2^{\frac{n-1}{2}}{n-1 \choose \frac{n-1}{2}} & n \text{ odd} \\
	0 & n \text{ even} \\
\end{dcases}.$$
\end{Thm}
A forthcoming paper will attempt to answer his additional question of what permutations satisfy $Q(w) = Q(w^{rc})$.

As for the structure of this paper, Section \ref{background} covers background, notation, and definitions needed throughout the paper.
Section \ref{theta} introduces a family of maps $\Phi_n$ and a function $\theta_n$ which map symmetric groups to ``neighboring'' symmetric groups.
These maps further allow us to prove Theorem \ref{main-theorem} in Section \ref{proof-conjectures}.

\noindent \textit{Acknowledgements.} We thank Nick Loehr, Jeremy Martin, Bruce Sagan, and Richard Stanley for their correspondence and insight on earlier drafts. 

\section{Background, Notation, and Definitions} \label{background}
%

As the RSK correspondence is a bijection from permutations to standard Young tableaux, we begin by fixing our notation for permutations.
We write elements of the symmetric group on $n$ letters in one-line notation so that 
$$\mathfrak{S}_n=\{(w_1,\dots,w_n)\in \mathbb{Z}^n \ : \ \{w_1,\dots,w_n\}=[n] \}.$$ 
In other words, $w$ represents the permutation 
$$\sigma = \begin{pmatrix}
1 & 2 & \cdots & n \\
w_1 & w_2 & \cdots & w_n
\end{pmatrix}.$$
For convenience and conciseness we write $w = w_1\dots w_n$ for an element of $\mathfrak{S}_n$.
We can now describe the method by which we turn a permutation into a pair of tableaux.

\begin{Def} \label{row-insertion}
Let $T$ be a column strict tableau and let $x$ be a positive integer.
Then the {\bf Schensted insertion} or {\bf row-insertion  algorithm} $T \leftarrow x$ is defined as follows:
\begin{itemize}
	\item If $T = \emptyset$, then $T \leftarrow x = \young(x)$.
	\item If $x \geq u$ for all entries $u$ in the top row of $T$, then append $x$ to the end of the top row of $T$.
	\item Otherwise, find the leftmost entry $u$ such that $x < u$. Replace $u$ with $x$, and then perform the row-insertion with $u$ in the subtableau consisting of the second and succeeding rows. In this case, we say $x$ {\bf bumps} $u$.
	\item Repeat until the bumping stops. 
\end{itemize}

To obtain the promised pair of standard Young tableaux, we let $P(w)$ be the {\bf insertion tableau} given by $((\emptyset \leftarrow w_1) \leftarrow w_2) \leftarrow \cdots \leftarrow w_n$ and let the {\bf recording tableau} $Q(w)$ be the standard tableau of the same shape as $P(w)$ that records where the new box appears in the underlying Young diagram at each step. 
The RSK correspondence is the map $w \mapsto (P(w),Q(w))$.
\end{Def}

That was all a little bit dense; let us look at an example.

\begin{Exa} \label{exa-RSK}
Consider $52314 \in \mathfrak{S}_5$.\\
\underline{Step 1:} The initial tableau is empty.
$$ P = \young(5) \qquad Q = \young(1)$$
\underline{Step 2:} 2 bumps 5.
$$ P = \young(2,5) \qquad Q = \young(1,2)$$
\underline{Step 3:} 3 appends to the first row.
$$ P = \young(23,5) \qquad Q = \young(13,2)$$
\underline{Step 4:} 1 bumps 2, 2 bumps 5.
$$ P = \young(13,2,5) \qquad Q = \young(13,2,4)$$
\underline{Step 5:} 4 appends to the first row.
$$ P = \young(134,2,5) \qquad Q = \young(135,2,4)$$
\end{Exa}

One can easily check that $w = 52314$ satisfies $Q(w) = Q(w^r)$.
We next define the evacuation process.

\begin{Def}[{\cite[Definition 3.7.2]{Sa}}] \label{def-jeu-de-taquin}
A forward \textbf{jeu de taquin slide} of a skew tableau of shape $\lambda / \mu$ is given by:

\begin{itemize}

\item Pick $\alpha$ to be an inner corner of $\mu$.

\item While $\alpha$ is not an inner corner of $\lambda$ do:

\begin{itemize}

\item If $\alpha = (i,j)$, let $\alpha'$ be the cell of min$\{P_{i+l,j}, P_{i,j+1}\}$.

\item Slide $P_{\alpha'}$ into cell $\alpha$ and let $\alpha := \alpha'$.

\end{itemize}

\end{itemize}

The resulting tableau is denoted $j^\alpha(P)$.
\end{Def}

\begin{Def}[{\cite[Definition 3.9.1]{Sa}}] \label{def-delta}
For any tableau $Q$, let $m$ be the minimal element of $Q$.
Then the \textbf{delta operator} applied to $Q$ yields a new tableau, $\Delta Q$, given by performing the following steps:
\begin{itemize}

\item Erase $m$ from its cell, $\alpha$, in $Q$.

\item Perform the slide $j^\alpha$ on the resultant tableau.

\end{itemize}
\end{Def}

\begin{Def}[{\cite[Definition 3.9.1]{Sa}}] \label{def-evacuation}
For any standard Young tableau $Q$ on $n$ elements, we define the \textbf{evacuation tableau}, $\epsilon(Q)$, as the vacating tableau for the sequence 
$$Q, \Delta Q, \Delta^2 Q, \dots, \Delta^n Q.$$
That is, the $d$th cell of $\epsilon(Q)$ contains $n - i$ if cell $d$ was vacated when passing from $\Delta^i Q$ to $\Delta^{i+1}Q$.
\end{Def}

Again, we return to our previous example to demonstrate evacuation.

\begin{Exa} \label{exa-eva}
Consider $Q(52314)$, which was calculated in Example \ref{exa-RSK}.\\
\underline{Step 1:} Cell (3,1) is vacated.
$$ \Delta Q = \young(235,4) \qquad \epsilon(Q) = \young(\hfil\hfil\hfil,\hfil,5)$$
\underline{Step 2:} Cell (1,3) is vacated.
$$ \Delta^2 Q = \young(35,4) \qquad \epsilon(Q) = \young(\hfil\hfil4,\hfil,5)$$
\underline{Step 3:} Cell (2,1) is vacated.
$$ \Delta^3 Q = \young(45) \qquad \epsilon(Q) = \young(\hfil \hfil 4,3,5)$$
\underline{Step 4:} Cell (1,2) is vacated.
$$ \Delta^4 Q = \young(5) \qquad \epsilon(Q) = \young(\hfil 24,3,5)$$
\underline{Step 5:} The final cell (1,1) is vacated.
$$ \Delta^5 Q = \emptyset \qquad \epsilon(Q) = \young(124,3,5)$$
\end{Exa}

Immediately, we see that $\epsilon(Q(w))^T = Q(w)$ for $w = 52314$, as expected from $Q(w) = Q(w^r)$.
To prove Theorem \ref{main-theorem}, we first construct three sets that will be used to great effect in Section 4.

\begin{Def} \label{def-R-H-M}
Define two sets of permutations:
$$R_n=\{ w\in \mathfrak{S}_n \ : \ Q(w)=Q(w^r)  \}$$
and 
$$H_n=\{ w\in \mathfrak{S}_n \ : \ Q(w)\text{ is of symmetric hook shape}  \},$$
where a symmetric standard Young tableau $T$ is one that shares the same shape with its transpose and a hook shape tableau has an underlying Young diagram of shape $\mu = (k,1^{n-k}), \ k < n$.
Note that $H_n$ is empty for all even $n$, as there are no symmetric hook shape tableaux for even $n$.

Swapping to sets of standard Young tableaux, for all shapes $\lambda \vdash n$, define the sets
$$M_n^\lambda= \{P \in \text{SYT}(\lambda) | \epsilon(P)^T = P\}.$$
\end{Def}

We now split Theorem \ref{main-theorem} into two parts.
Theorem \ref{hooks}, combined with the fact that $H_n$ is empty for all even $n$, sheds light on some of the subtleties of the main result, specifically the dependence on the parity of $n$.

\begin{Thm} \label{hooks}
For all $n \geq 1$, the set $R_n$ is a subset of $H_n$, forcing $Q(w)$ to have a symmetric hook shape for all $w \in R_n$.
\end{Thm}

Finally, Theorem \ref{count-stuff} --- previously a conjecture --- is what began our research direction.

\begin{Thm} \label{count-stuff}
The cardinality of the set $R_n$ is given by
$$|R_n| =
\begin{dcases}
	2^{\frac{n-1}{2}}{n-1 \choose \frac{n-1}{2}} & n \text{ odd} \\
	0 & n \text{ even} \\
\end{dcases}.$$
\end{Thm}

\section{The Map(s) $\phi_{a,b}$ and Their Left Inverse} \label{theta}

To begin proving Theorem~\ref{hooks}, we define a family of maps that take permutations on $n$ letters to permutations on $n+2$ letters.

\begin{Def} \label{def-phi-a-b}
Pick two elements $a, b \in [n+2]$ such that $a \neq b$.
Let $c = \min\{a,b\}$ and $d = \max\{a,b\}$.
Then we define the map $\phi_{a,b} : \mathfrak{S}_n \mapsto\mathfrak{S}_{n+2}$, given by 
\begin{equation*}
	\phi_{a,b}(w)_{i+1} = \begin{cases}
		a & \text{ if $i +1 = 1$} \\
		w_i & \text{ if $w_i < c$}\\
		w_i + 1 & \text{ if $c \leq w_i < d-1$}\\
		w_i + 2 & \text{ if $d-1 \leq w_i$} \\
		b & \text{ if $i +1 = n+2$} 
	\end{cases}
\end{equation*}
for all $i+1 \in [n+2]$.
\end{Def}

We first show that the maps are well-defined.

\begin{Lem} \label{phi-well-defined}
If $a \neq b$, then the map $\phi_{a,b} : \mathfrak{S}_n \mapsto\mathfrak{S}_{n+2}$ is well-defined for all $a,b \in [n+2]$.
\end{Lem}

\begin{proof}

First, for $i, j \in [n]$ we can clearly see that $w_i < w_j$ implies that $\phi_{a,b}(w)_{i+1} < \phi_{a,b}(w)_{j+1}$.
Hence $\phi_{a,b}(w)_{i+1} = \phi_{a,b}(w)_{j+1}$ if and only if $w_i = w_j$ for any $i,j \in [n]$.
If $w_i < c$, then $\phi_{a,b}(w)_{i+1} < c$.
If $c \leq w_i < d - 1$, then $c < \phi_{a,b}(w)_{i+1} < d$.
If $d-1 \leq w_i$, then $d < \phi_{a,b}(w)_{i+1}$.
Thus 
$$\phi_{a,b}(w)_1 \neq \phi_{a,b}(w)_{n+2},$$
$$ \phi_{a,b}(w)_{i+1} \neq \phi_{a,b}(w)_{1} ,$$
$$\phi_{a,b}(w)_{i+1} \neq \phi_{a,b}(w)_{n+2} ,$$
$$\text{and } \phi_{a,b}(w)_{i+1} \neq \phi_{a,b}(w)_{j+1}$$
for any $i,j \in [n]$ with $i \neq j$.
Therefore, the map $\phi_{a,b}$ is a well-defined set function from $\mathfrak{S}_n$ to $\mathfrak{S}_{n+2}$.

\end{proof}

\begin{Exa} \label{exa-phi}
We list some examples of the action of different $\phi_{a,b}$ on $52314 \in \mathfrak{S}_5$.
\begin{itemize}

\item $\phi_{1,2}(52314) = 1745362$

\item $\phi_{1,7}(52314) = 1634257$

\item $\phi_{5,3}(52314) = 5724163$

\item $\phi_{3,5}(52314) = 3724165$

\end{itemize}

As we saw previously, the permutation $52314$ belongs to $R_5$.
The permutation $\phi_{1,2}(52314)$ belongs to $R_7$, but $\phi_{1,7}(52314)$ does not.
This relation between $R_n$ and $R_{n+2}$ is explored in Lemma \ref{phi-preserves-R}.

\end{Exa}

Now that we have seen the maps in action, we can speak more of their properties.

\begin{Rmk} \label{rmk-phi}

Let $\Phi_n$ denote the family of maps $\phi_{a,b}$ from $\mathfrak{S}_n$ to $\mathfrak{S}_{n+2}$.
\begin{itemize}
	\item Each map $\phi_{a,b}$ is injective.
	
	\item In general, the maps $\phi_{a,b}$ are not group homomorphisms, as $\phi_{a,b}$ only sends the identity element to the identity element when $a = 1$ and $b = n+2$. However, the map $\phi_{1, n+2}$ is a group monomorphism. 
	
	\item If $\phi_{a,b}(w) = \phi_{x,y}(w)$ for any $w \in \mathfrak{S}_n$, then $a = x$ and $b = y$.
	
	\item The set $\mathfrak{S}_{n+2}$ is partitioned by the images of the functions $\phi_{a,b}$.
	
	\item If $w_i < w_j$ for $i,j \in [n]$, then $\phi_{a,b}(w)_{i+1} < \phi_{a,b}(w)_{j+1}$ from the proof of Lemma \ref{phi-well-defined}, meaning all the relative orderings of $w$ are preserved.
	
\end{itemize}

\end{Rmk}

The $\Phi_n$ maps go ``up'' the symmetric groups, but we need a function to go ``down'' from $\mathfrak{S}_{n+2}$ to $\mathfrak{S}_n$.
We then construct a left inverse function shared by every $\phi_{a,b} \in \Phi_n$.

\begin{Def} \label{def-theta}
We define the map $\theta_n : \mathfrak{S}_{n+2} \mapsto\mathfrak{S}_n$, given by
\begin{equation*}
	\theta_n(w)_i = \begin{cases}
		w_{i+1} & \text{ if $w_{i+1} < c$} \\
		w_{i+1} - 1 & \text{ if $c < w_{i+1} < d$} \\
		w_{i+1} - 2 & \text{ if $d < w_{i+1}$}
	\end{cases}
\end{equation*}
for all $i \in [n]$, where $c = \min\{w_1, w_{n+2}\}$ and $d = \max\{w_1, w_{n+2}\}$.
\end{Def}

As before, we will first show that $\theta_n$ is well-defined.

\begin{Lem} \label{theta-well-defined}
The map $\theta_n : \mathfrak{S}_{n+2} \mapsto \mathfrak{S}_n$ is well-defined.
More strongly, if $i,j \in [n]$, then $w_{i+1} < w_{j+1}$ implies that $\theta_n(w)_i < \theta_n(w)_j$.

\end{Lem}

\begin{proof}

First, we show that $\theta_n(w)_i \in [n]$ for all $i \in [n]$.
If $\theta_n(w)_i < 1$, then either $w_{i+1} =1$ or $w_{i+1} = 2$.
However, when $w_{i+1} = 1$, we must have $\theta_n(w)_i = 1$.
When $w_{i+1} = 2$, then $\theta_n(w)_i = 1$ or $2$, contradicting our assumption of $\theta_n(w)_i < 1$.
If $\theta_n(w)_i > n$, then either $w_{i+1} =n+1$ or $w_{i+1} = n+2$.
However, when $w_{i+1} = n+2$, we must have $\theta_n(w)_i = n$.
When $w_{i+1} =n+1$, then $\theta_n(w)_i = n-1$ or $n$, contradicting our assumption of $\theta_n(w)_i > n$.
Thus $\theta_n(w)_i \in [n]$ for all $i \in [n]$.

Now, if we show that $w_{i+1} < w_{j+1}$ implies $\theta_n(w)_i < \theta_n(w)_j$ for all $i,j \in [n]$, then this will prove that $\theta_n(w) \in \mathfrak{S}_n$.
Suppose $w_{i+1} < w_{j+1}$.
If 
$$w_{j+1} < c, \, c < w_{i+1} < w_{j+1} < d, \text{ or } d < w_{i+1},$$ 
then $\theta_n(w)_i < \theta_n(w)_j$, where $c = \min\{w_1,w_{n+2}\}$ and $d = \max\{w_1,w_{n+2}\}$.
If 
$$w_{i+1} < c < d < w_{j+1},$$ 
then $w_{j+1} -2 > w_{i+1}$, forcing $\theta_n(w)_i < \theta_n(w)_j$.
If 
$$c < w_{i+1} < d < w_{j+1},$$
then $w_{j+1} -2 > w_{i+1} -1$, forcing $\theta_n(w)_i < \theta_n(w)_j$.
Thus $w_{i+1} < w_{j+1}$ implies $\theta_n(w)_i < \theta_n(w)_j$ for all $i,j \in [n]$.
Therefore, $\theta_n(w)$ is a well-defined set function from $\mathfrak{S}_{n+2}\mapsto \mathfrak{S}_n$.

\end{proof}

\begin{Exa} \label{exa-theta}
Some examples of the action of $\theta_n$ are listed below.

\begin{itemize}

\item $\theta_1(231) = 1$

\item $\theta_3(52314) = 231$

\item $\theta_5(1634257) = 52314$

\end{itemize}

Note that $\theta_n$ seems to send $R_{n+2}$ to $R_n$, as every permutation, $w$, above has $Q(w) = Q(w^r)$.
This relation is further explored in Lemma \ref{theta-preserves-R}.
\end{Exa}

We can now prove that $\theta_n$ acts as a left inverse for each $\phi_{a,b} \in \Phi_n$.

\begin{Lem} \label{theta-inverse}
For all $\phi_{a,b} \in \Phi_n$ and $w \in \mathfrak{S}_n$,
$$\theta_n(\phi_{a,b}(w)) = w.$$
\end{Lem}

\begin{proof}

From the definition of $\theta_n$, 
\begin{equation*}
	\theta_n(\phi_{a,b}(w))_i = \begin{cases}
		\phi_{a,b}(w)_{i+1} & \text{ if $\phi_{a,b}(w)_{i+1} < c$} \\
		\phi_{a,b}(w)_{i+1} - 1 & \text{ if $c < \phi_{a,b}(w)_{i+1} < d$} \\
		\phi_{a,b}(w)_{i+1} - 2 & \text{ if $d < \phi_{a,b}(w)_{i+1}$}
	\end{cases}
\end{equation*}
for all $i \in [n]$, where $c = \min\{a,b\}$ and $d = \max\{a,b\}$.
If $\phi_{a,b}(w)_{i+1} < c$, then $w_i < c$.
If $c < \phi_{a,b}(w)_{i+1} < d$, then $c \leq w_i < d$.
If $d < \phi_{a,b}(w)_{i+1}$, then $d-1 \leq w_i$.
It follows from the definition of $\phi_{a,b}$ that
\begin{equation*}
	\theta_n(\phi_{a,b}(w))_i = \begin{cases}
		w_i & \text{ if $w_i < c$} \\
		w_i & \text{ if $c \leq w_i < d$} \\
		w_i & \text{ if $d-1 \leq w_i$}
	\end{cases}
\end{equation*}
for all $i \in [n]$.
Therefore, for all $\phi_{a,b} \in \Phi_n$,
$$\theta_n(\phi_{a,b}(w)) = w.$$
\end{proof}

Because $\theta_n$ is a left inverse and the images of the $\Phi_n$ maps partition $\mathfrak{S}_{n+2}$, we immediately see that
$$\theta_n^{-1}(w) = \bigcup_{a,b \in [n+2]; \, a \neq b} \phi_{a,b}(w).$$  
Combined with the next two lemmas, this fact is instrumental in proving Theorem \ref{hooks}.

\begin{Lem} \label{theta-r-c}
For any $w \in \mathfrak{S}_{n+2}$,
$$\theta_n(w^r) = \theta_n(w)^r$$
and
$$\theta_n(w^c) = \theta_n(w)^c.$$
\end{Lem}

\begin{proof}

From the definition of $\theta_n$,
\begin{equation*}
	\theta_n(w^r)_i = \begin{cases}
		w_{n-i+2} & \text{ if $w_{n-i+2} < c$} \\
		w_{n-i+2} - 1 & \text{ if $c < w_{n-i+2} < d$} \\
		w_{n-i+2} - 2 & \text{ if $d <w_{n-i+2}$}
	\end{cases},
\end{equation*}
where $c = \min\{a,b\}$ and $d = \max\{a,b\}$.
Additionally,
\begin{equation*}
	\theta_n(w)_i^r= \begin{cases}
		w_{n-i+2} & \text{ if $w_{n-i+2} < c$} \\
		w_{n-i+2} - 1 & \text{ if $c < w_{n-i+2} < d$} \\
		w_{n-i+2} - 2 & \text{ if $d < w_{n-i+2}$}
	\end{cases}.
\end{equation*}
Thus $\theta_n(w^r)_i = \theta_n(w)_i^r$ for all $i \in [n]$.
Therefore,
$$\theta_n(w^r) = \theta_n(w)^r.$$

Again,
\begin{equation*}
	\theta_n(w^c)_i = \begin{cases}
		n-w_{i+1}+3 & \text{ if $n-w_{i+1}+3 < n-d+3$} \\
		n-w_{i+1}+2 & \text{ if $n-d+3 < n-w_{i+1}+3 < n-c+3$} \\
		n-w_{i+1}+1 & \text{ if $n-c+3 < n-w_{i+1}+3$}
	\end{cases}.
\end{equation*}
We also have
\begin{equation*}
	\theta_n(w)_i^c = \begin{cases}
		n-w_{i+1}+1 & \text{ if $w_{i+1} < c$} \\
		n-w_{i+1}+2 & \text{ if $c <w_{i+1}< d$} \\
		n-w_{i+1}+3 & \text{ if $d < w_{i+1}$}
	\end{cases}.
\end{equation*}
If 
$$n-w_{i+1}+3 < n-d+3,$$ 
then $d < w_{i+1}$.
If 
$$n-d+3 < n-w_{i+1}+3 < n-c+3, $$
then $c < w_{i+1} < d$.
If $$n-c+3 < n-w_{i+1}+3,$$ 
then $w_{i+1} < c$.
Thus $\theta_n(w^c)_i  = \theta_n(w)_i^c$ for all $i \in [n]$.
Therefore,
$$\theta_n(w^c) = \theta_n(w)^c.$$

\end{proof}

\begin{Lem} \label{phi-preserves-symmetry}
Choose $\phi_{a,b} \in \Phi_n$ and $w \in H_n$ for odd $n$.
Then either $\phi_{a,b}(w) \in H_{n+2}$ or $Q(\phi_{a,b}(w))$ is not a symmetric tableau.
\end{Lem}

\begin{proof}
Let $w \in H_n$.
Suppose that $\phi_{a,b}(w) \notin H_{n+2}$ for some $\phi_{a,b} \in \Phi_n$.
As $\phi_{a,b}$ preserves all the relative orderings of $w$ by Remark \ref{rmk-phi}, the longest increasing and longest decreasing subsequences of $\phi_{a,b}(w)$ are the same length or longer than the longest increasing and longest decreasing subsequences of $w$.
It is well-known that the first row and the first column of $Q(w)$ share their lengths with the longest increasing and the longest decreasing subsequenes \cite[Theorem 3.3.2]{Sa}.
Hence, the longest subsequences of $\phi_{a,b}(w)$ must have length at least $\frac{n+1}{2}$ by the symmetry of $Q(w)$.
If $Q(\phi_{a,b}(w))$ is symmetric, then the length of the longest increasing and the longest decreasing subsequences of $\phi_{a,b}(w)$ must equal each other. 
Since $\phi_{a,b}(w) \notin H_{n+2}$, the longest increasing and decreasing subsequences necessarily have length $\frac{n+1}{2}$.
Thus two cells were added to the second row or column of $Q(\phi_{a,b}(w))$, contradicting the symmetry of $Q(\phi_{a,b}(w))$.
Therefore, either $\phi_{a,b}(w) \in H_{n+2}$ or $Q(\phi_{a,b}(w))$ is not a symmetric tableau.
 
\end{proof}

As an example of the possibilities discussed in Lemma~\ref{phi-preserves-symmetry}, we give the following.

\begin{Exa}
Take this Young tableau of symmetric hook shape: 
$$\yng(3,1,1).$$
After an application of a $\phi_{a,b} \in \Phi_n$, two new cells will be added to the diagram. 
There are four possibilities --- up to transposition --- for the resulting shape:
$$\yng(4,1,1,1) \qquad \yng(4,2,1) \qquad \yng(5,1,1) \qquad \yng(3,3,1).$$
Clearly, only the tableau in $H_7$ is symmetric.
\end{Exa}

\section{Proof of Theorems \ref{hooks} and \ref{count-stuff}} \label{proof-conjectures}

To begin our proof of Theorem \ref{hooks}, we first prove the two lemmas we mentioned previously when exploring the actions of the $\Phi_n$ and $\theta_n$ maps.

\begin{Lem} \label{theta-preserves-R}
If $w \in R_{n+2}$, then $\theta_n(w) \in R_n$.
\end{Lem}

\begin{proof}

Suppose that $w \in R_{n+2}$.
Let $x = w_2 w_3 \dots w_{n+2}$ and $y = w_{n+1}w_{n}\dots w_1$.
Then $Q(x) = \Delta Q(w)$ and $Q(y) = \Delta Q(w^r)$  \cite[Proposition 3.9.3]{Sa}.
As $w \in R_{n+2}$, we have that $Q(w) = Q(w^r)$, showing that $Q(x) = Q(y)$.
Since $n+1$ occupies the same cell in $Q(x)$ and $Q(y)$, not inserting $w_{n+2}$ or $w_1$ in the respective $x$ and $y$ cases produces the same recording tableau.
Hence,
$$Q(w_2 w_3 \dots w_{n+1}) = Q(w_{n+1} w_{n} \dots w_2).$$

As $\theta_n$ preserves all the relative orderings among middle entries of permutations by Lemma \ref{theta-well-defined}, it follows that
$$Q(\theta_n(w)) = Q(w_2 w_3 \dots w_{n+1})$$
and
$$ Q(\theta_n(w^r)) = Q(w_{n+1} w_{n} \dots w_2).$$
As Lemma \ref{theta-r-c} proves $Q(\theta_n(w)^r) = Q(\theta_n(w^r)),$ the equality 
$$Q(\theta_n(w)) = Q(\theta_n(w)^r)$$
holds.
Therefore, the permutation $\theta_n(w)$ belongs to $R_n$ for all $w \in R_{n+2}$.

\end{proof}

\begin{Lem} \label{phi-preserves-R}
For all $n$, we have that
$$R_{n+2} \subseteq \bigcup_{\phi_{a,b} \in \Phi_n} \phi_{a,b}(R_{n}).$$
\end{Lem}

\begin{proof}

As $w = \phi_{a,b}(\theta_n(w))$ for all $w \in R_{n+2}$ and some $\phi_{a,b} \in \Phi_n$ dependent on $w$, Lemma \ref{theta-preserves-R} gives us that $\theta_n(w) \in R_n$.
Thus $w \in \phi_{a,b}(R_n)$.
Hence,
$$R_{n+2} \subseteq \bigcup_{\phi_{a,b} \in \Phi_n} \phi_{a,b}(R_{n}).$$

\end{proof}

\begin{Thm}[Proof of Theorem \ref{hooks}] \label{proof-conj-hooks}
For all $n \geq 1$, the set $R_n$ is a subset of $H_n$, forcing $Q(w)$ to have a symmetric hook shape for all $w \in R_n$.
\end{Thm}

\begin{proof}

We split into two separate induction arguments, one odd and one even.
In the odd scenario, the base case of $n = 0 + 1$ is clear, as $R_1 = H_1 = \mathfrak{S}_1$.
As such, we assume that $R_{n-2} \subseteq H_{n-2}$ for $n = 2\ell + 1$, where $\ell$ is the integer we induct on.
If $v \in R_n$, then $P(v) = P(v^r)^T$.
Hence the shape of $Q(v)$ must be symmetric for all $v\in R_n$.
If $\phi_{a,b}(w)$ is not in $H_n$ for a $\phi_{a,b} \in \Phi_{n-2}$ and some $w \in R_{n-2}$, then the shape of $Q(\phi_{a,b}(w))$ is non-symmetric by Lemma \ref{phi-preserves-symmetry}.
This further implies that $\phi_{a,b}(w)$ is not in $R_n$.
The contrapositive gives that $\phi_{a,b}(w) \in R_n$ implies that $\phi_{a,b}(w) \in H_n$.

From Lemma \ref{phi-preserves-R}, we have that $R_{n} \subseteq \bigcup_{\phi_{a,b} \in \Phi_{n-2}} \phi_{a,b}(R_{n-2}).$
Thus every element of $R_n$ is of the form $\phi_{a,b}(w)$ for $\phi_{a,b} \in \Phi_{n-2}$ and $w \in R_{n-2}$.
Therefore $R_n \subseteq H_n$, completing the inductive step.
Induction on $\ell$ gives us the $R_n \subseteq H_n$ for all odd $n$.

In the even case, the base case of $n = 2$ is clear, as $R_2 = H_2 = \emptyset$.
The set $H_n$ is empty for all even $n$.
As such, we assume that $R_{n-2}$ is empty for $ n = 2\ell$, where $\ell$ is the integer we induct on.
If $R_n$ is non-empty, Lemma \ref{theta-preserves-R} gives us that $\theta_{n-2}(R_n) \subseteq R_{n-2}$, creating a contradiction of our inductive assumption.
Hence $R_n$ must be empty and a subset of $H_n$ as well, completing the inductive step.
Therefore $R_n \subseteq H_n$ for all $n$.

\end{proof}

\begin{Cor} \label{M-empty}

For all $\lambda \vdash n$, the set $M_n^{\lambda}$ is non-empty if and only if $\lambda$ is a symmetric hook shape.

\end{Cor} 

\begin{Rmk} \label{rmk-hook}
It should be stressed that $R_n$ really is a \textit{proper subset} of $H_n$ for odd $n \geq 5$. 
In other words, there are $w \in \mathfrak{S}_n$ that are of symmetric hook shape but do not have fixed recording tableaux under the reverse operation. 
The permutation $34521$ is such an example.
\end{Rmk}

Since we have shown that $M_n^{\lambda}$ is only non-empty when $n$ is odd and $\lambda$ is a symmetric hook shape, we may restrict ourselves to those conditions and begin to examine the properties of $P \in M_n^{\lambda}$.

\begin{Lem} \label{first-row-M}
Let $\lambda = (\frac{{n + 1}}{2},1^{\frac{{n - 1}}{2}})$.
For odd $n$, the standard Young tableau $P$ belongs to $M_n^{\lambda}$ if and only if $i$ in the first row of $P$ implies that $n-i+2$ belongs to the first column of $P$ for all $i \in [n]$ with $i > 1$.
\end{Lem}
\begin{proof}

Suppose that $P \in M_n^{\lambda}$. 
Additionally assume that $i$ is in the first row of a symmetric hook shape tableau, $P$, where $i \in [n]$ and $i > 1$.
The cell vacated when passing from $\Delta^{i-2} P$ to $\Delta^{i-1} P$ is filled with $n-i+2$ in $\epsilon(P)$.
As $i-1$ is minimal in $\Delta^{i-2} P$, it appears in the first cell of $\Delta^{i-2} P$ and $i$ is in the cell to its right.
Hence the cell vacated when passing from $\Delta^{i-2} P$ to $\Delta^{i-1} P$ will be in the first row.
Thus $n-i+2$ appears in the first row of $\epsilon(P)$ if $i$ appears in the first row of $P$.
As $P = \epsilon(P)^T$, the location of $i$ in the first row of $P$ implies that $n-i+2$ belongs to the first column of $P$ for all $i \in [n]$ with $i > 1$ whenever $P \in M_n^{\lambda}$.

Conversely, suppose that $i$ in the first row of $P$ implies that $n-i+2$ belongs to the first column of $P$ for all $i \in [n]$ with $i > 1$.
Let $1,i_2, i_3, \dots, i_{\frac{n+1}{2}}$ be the increasing sequence that forms the first row of $P$.
Thus $1, n-i_{\frac{n+1}{2}}+2, \dots, n-i_3+2, n-i_2+2$ must be the increasing sequence that forms the first column of $P$ and the first row of $\epsilon(P)$.
Similarly, the first row of $P$ and the first column of $\epsilon(P)$ must coincide.
Hence $P = \epsilon(P)^T$ whenever $i$ in the first row of $P$ implies that $n-i+2$ belongs to the first column of $P$ for all $i \in [n]$ with $i > 1$.
Therefore, for odd $n$, the standard Young tableau $P$ belongs to $M_n^{\lambda}$ if and only if $i$ in the first row of $P$ implies that $n-i+2$ belongs to the first column of $P$ for all $i \in [n]$ with $i > 1$.

\end{proof}

Returning to the previous Example \ref{exa-RSK}, we can verify that the tableau $Q(52314) \in M_5^{\lambda}$ by Lemma \ref{first-row-M}.
We now prove a final lemma that will allow us to correctly count the number of tableaux in $M_n^{\lambda}$.
This lemma will then give us the desired proof of Theorem \ref{count-stuff}.

\begin{Lem} \label{size-M}
Let $\lambda = (\frac{{n + 1}}{2},1^{\frac{{n - 1}}{2}})$.
If $n$ is odd, then
$$|M_n^{\lambda}| = 2^{\frac{n-1}{2}}.$$
\end{Lem}

\begin{proof}

For every $i$ such that $2 \leq i \leq \frac{n+1}{2}$, the symmetric hook shape forces $i$ to be placed in the first row or the first column of a tableau.
There are then $2^{\frac{n-1}{2}}$ ways for the first $\frac{n-1}{2}$ integers greater than 1 to be placed in the first row or column.
As $i$ in the first row (or column) implies that $n-i+2$ is in the first column (or row) by Lemma \ref{first-row-M}, we have that every tableau in $M_n^{\lambda}$ is uniquely determined by where the first $\frac{n-1}{2}$ integers greater than 1 are placed.
Hence, the cardinality of $M_n^{\lambda}$ is $2^{\frac{n-1}{2}}$.

\end{proof}

\begin{Thm}[Proof of Theorem \ref{count-stuff}] \label{proof-conj-count}
The cardinality of the set $R_n$ is given by
$$|R_n| =
\begin{dcases}
	2^{\frac{n-1}{2}}{n-1 \choose \frac{n-1}{2}} & n \text{ odd} \\
	0 & n \text{ even} \\
\end{dcases}.$$
\end{Thm}

\begin{proof}

For even $n$, the result follows directly from Theorem \ref{proof-conj-hooks} as $R_n$ is empty.
We then focus on odd $n$.
Let $\lambda = (\frac{{n + 1}}{2},1^{\frac{{n - 1}}{2}})$.
If $w \in R_n$, then $Q(w) = Q(w^r) = \epsilon(Q(w))^T$ \cite[Theorem 3.9.4]{Sa}.
Thus $Q(w) \in M_n^{\lambda}$ for all $w \in R_n$.

Let $P \in  M_n^{\lambda}$.
The number of pairs of standard Young tableaux with $P$ as the recording tableau is equal to the number of standard Young tableaux of shape $\lambda$.
As the RSK correspondence is a bijection of permutations and pairs of standard Young tableaux, there are $f^{\lambda} = {n-1 \choose \frac{n-1}{2}}$ many $w \in \mathfrak{S}_n$ such that $Q(w) = P$.
If $Q(w) \in M_n^{\lambda}$, then $Q(w) = \epsilon(Q(w))^T =  Q(w^r)$.
Thus $w \in R_n$.
The cardinality of $R_n$ must then be $|M_n^{\lambda}|f^{\lambda}$ for all odd $n$.

Therefore,
$$|R_n| =
\begin{dcases}
	2^{\frac{n-1}{2}}{n-1 \choose \frac{n-1}{2}} & n \text{ odd} \\
	0 & n \text{ even} \\
\end{dcases}.$$
\end{proof}
Theorems \ref{proof-conj-hooks} and \ref{proof-conj-count}, along with Lemma \ref{first-row-M}, combine to prove Theorem \ref{main-theorem}.

\end{document}